\documentclass[11pt]{article}
\usepackage{amsfonts,amsmath,amssymb,amsthm, amscd, subcaption}
\usepackage{graphicx}
\numberwithin{equation}{section}

\newtheorem{theorem}{Theorem}[section]
\newtheorem{prop}[theorem]{Proposition}

\theoremstyle{definition}

\newtheorem{example}[theorem]{Example}

\def\<{{\langle}}
\def\>{{\rangle}}

\def\g{{\gamma}}

\def\w{{\mathbb T}}

\def\w{\wau}

\def\w{\omega}

\def\ni{\noindent} 

\begin{document}

\title{TANGLES AND LINKS: A VIEW WITH TREES}

\author{Daniel S. Silver 
\and Susan G. Williams\thanks {The authors are grateful for the support of the Simons Foundation.} }

\maketitle 
\begin{quote} \textit{The clearest way into the Universe is through a forest wilderness.} --  John Muir (1890)
\end{quote}
\bigskip
\begin{abstract} A result about spanning forests for graphs yields a short proof of Krebes's theorem concerning embedded tangles in links. \end{abstract}

\section{Introdution} \label{intro}  Let $G$ be a finite graph, possibly with multiple edges. We denote the vertex and edge sets by $V_G$ and $E_G$, respectively.  For $e \in E_G$, let $w_e \in R$ be \emph{edge weights}, elements of a commutative ring $R$. The \emph{(tree) weight} of $G$ is
$$\w_G =  \sum_T \prod_{e \in E_T} w_e,$$
where the summation is taken over all spanning trees of $G$.  When all edge weights $w_e$ are equal to $1$, the weight of $G$ is simply the number of spanning trees of the graph. If $G$ is not connected then its weight is $0$.

More generally, if $\g, \g'$ are subsets of $V_G$ of equal cardinality, then we define
$$\w_{G, \g, \g'} = \sum_F \prod_{e \in E_F} w_e,$$
where the summation is taken over all spanning forests $F$ consisting of trees with exactly one vertex in $\g$ and one vertex in $\g'$. The weight $\w_{G, \g, \g'}$ is easily computed from a Laplacian matrix of $G$ (see Section \ref{krebessection}).

When $\g = \g'$, the vertices in $\g$ are contained in separate trees of $F$. In this case we say that $F$ is \emph{rooted at} $\g$, and we  
shorten the notation $\w_{G, \g, \g}$ to $w_{G, \g}$.  Note that $\w_{G, \g}$  is simply $\w_G$ whenever $\g$ is a singleton.

We say $G$ is the \emph{join} of graphs $H, K$ if $G=H\cup K$ and $H\cap K$ is a single vertex. In this case, clearly $\w_G = \w_H \cdot \w_K$. The next proposition generalizes this observation.

\begin{prop}\label{main} Suppose $G$ is the union of two subgraphs $H, K$ such that $H \cap K = \g$, where $\g = \{1, 2, \ldots, n\}   \subset V_G$. 
\item{(i)} If $n =2$, then $\w_G= \w_H \cdot \w_{K, \g} + \w_{H, \g} \cdot \w_K$. 
\item{(ii)} If $n=3$, then $\w_G=$  \vspace{-.8 em} $$\w_H \cdot \w_{K, \g}+\w_{H, \{1,2\}} \cdot \w_{K,{\{1,3\},\{2,3\}}} + \w_{H, \{2,3\}} \cdot \w_{K,{\{1,2\},\{1,3\}}} +\w_{H, \{1,3\}} \cdot \w_{K,{\{1,2\},\{2,3\}}}+\w_{H, \g} \cdot \w_K.$$

\end{prop}  

\begin{proof} (i) The union of a spanning tree of $K$ with a spanning forest of $H$ rooted at $\g$ is a spanning tree of $G$, as is also the union of a spanning tree of $H$ with a spanning forest of $K$ rooted at $\g$.  Conversely, every spanning tree $T$ of $G$ arises in exactly one way as a union of one of theses two types, depending on whether $T\cap K$ is connected or not.  The result follows by summing over these two cases separately in the definition of $\w_G$.

(ii) In a similar vein, consider a spanning tree $T$ of $G$. If $T \cap K$ is connected, then $T \cap H$ is a forest of three component trees, and vertices $1, 2, 3$ are in separate components.

Suppose $T \cap K$ has two components.  Then two elements of $\g$, say $1$ and $2$, lie in one component, and $3$ in the other.  Such a forest can be characterized as consisting of trees with exactly one vertex in each of the sets $\{1,3\}$, $\{2,3\}$. Now $T \cap H$ must be a spanning forest of $H$ rooted at $\{1,2\}$. There are two other analogous cases.

Finally, if $1,2,3$ lie in separate components of $T \cap K$, then $T \cap H$ must be a spanning tree of $H$.  
In each of these five cases we see that conversely, taking a union of spanning forests of the types described gives a spanning tree of $G$.  The result follows by summing over cases.
\end{proof} 

\section{Embedded 4-tangles in links}  \label{krebessection} A \emph{knot} is a circle smoothly embedded in the 3-sphere. More generally, a \emph{link} $\ell$ is a  finite collection of pairwise disjoint knots. Two links are equivalent if there is an  isotopy of $S^3$ that takes one link to the other. The simplest sort of link, the union of pairwise disjoint circles in the plane, said to be \emph{trivial}. Other links are called \emph{nontrivial}.

One attraction of knot theory is in the fact that any link $\ell$ can be described by a drawing or \emph{diagram}, a generic projection $D$ in the plane with a \textit{trompe l'oeil} device at each crossing indicating how one strand of the link passes over another. An example is seen in Figure \ref{knot}.

Tangles are close relatives of links. A \emph{$2n$-tangle} $t$, for $n$ a positive integer, consists of $n$ disjoint arcs and any finite number of simple closed curves properly embedded in the $3$-ball. The endpoints of the arcs are constrained to meet the ball's boundary along a great circle in $2n$ specified points. A pair of $2n$-tangles are equivalent if there is an isotopy of the ball, fixing points of the boundary, taking one $2n$-tangle to the other. 

Any $2n$-tangle can be represented by a diagram, similar to a link diagram, in a disk with endpoints on the boundary, $n$ points above and $n$ below. We can find $2n$-tangles in any link diagram $D$: a circle meeting $D$ in general position encloses a $2n$-tangle diagram, for some $n$. When this happens we say that $t$ \emph{embeds in} $\ell$. 

By a \emph{tangle} we mean a $4$-tangle. Any tangle diagram can be closed to a link $n(t)$ or $d(t)$ by connecting free ends in pairs of disjoint embedded arcs outside the diagram, as in Figure \ref{numden}. These are referred to as the \emph{numerator} and \emph{denominator} closures of $t$. In general, the number of closures of a $2n$-tangle diagram is the $n$th Catalan number  $C_n= \frac{(2n)!}{n! (n+1)!}$. 

\begin{figure}
\begin{center}
\includegraphics[height=1.2 in]{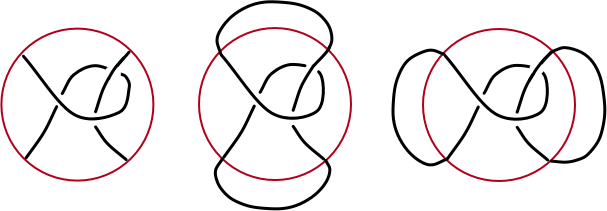}
\caption{Tangle $t$ (left); numerator closure $n(t)$ (center);  denominator closure $d(t)$ (right)}
\label{numden}
\end{center}
\end{figure}

If a knot diagram $D$ has many crossings, then it can be a formidable task to ascertain that the knot is nontrivial. One might ask if it is possible to decide based upon the discovery of a particular type of $2n$-tangle embedded in the link. When $n=1$, the problem was solved by H. Schubert,  who showed in 1949  \cite{BZ03} that any knot containing a nontrivial $2$-tangle (a ``local knot")  is itself nontrivial. 

In his Ph.D. dissertation \cite{K99} D. Krebes showed how a nontrivial embedded tangle may force a knot to be nontrivial. Krebes employed the \emph{determinant} ${\rm Det}(\ell)$ of a link, a well-known \emph{numerical link invariant}, a numerical quantity that depends only on the link and not the diagram considered.

The determinant of a link $\ell$ has several equivalent definitions. The most elementary begins with a checkerboard shaded diagram $D$ of the link and associates a  \emph{Tait graph} $G$: vertices correspond to shaded faces; edges correspond to crossings, joining shaded faces that share a crossing and labeled with weights $+1$ or $-1$ depending on the sense of the crossing (Figure \ref{crossing}). (Unshaded faces can be used in place of shaded ones, resulting in a dual Tait graph for $\ell$.)  Next one builds a matrix $M$ with rows and columns indexed by vertices: the $i$th diagonal entry is the sum of weights of non-loop edges incident to the $i$th vertex; for $i \ne j$ the $i,j$th entry is $-1$ times the sum of weights of edges joining the $i$th vertex to the $j$th. Knot theorists recognize $M$ as an \emph{unreduced Goeritz matrix} of the link $\ell$. The principal minors of $M$ are equal and ${\rm Det}(\ell)$ is defined to be their absolute value  (see \cite{BZ03}). 

\begin{example} \label{ex} The left-hand side of Figure \ref{knot} is a diagram of a knot $k$.
In the center is a checkerboard-shaded diagram with shaded faces numbered, and on the right is the associated Tait graph $G$. (Unlabeled edges are taken to have weight 1.) The unreduced Goeritz matrix is $$M = \begin{pmatrix} 
0 & 1 & -1 & -1 & 0 & 0 & 1\\
1 & 0 & 0 & 0 & -1&-1& 1 \\
-1 & 0 & 2 & 0 & -1 & 0 & 0\\
-1 & 0 & 0 & 2 & 0 & -1 & 0 \\
0 & -1 & -1 & 0 & 0 & 2 & 0 \\
0 & -1 & 0 & -1 & 2 & 0 & 0 \\
1 & 1 & 0 & 0 & 0 & 0 & -2 
\end{pmatrix}$$
\noindent Any principal minor is 25, the determinant of $k$.

\end{example}

\begin{figure}
\begin{center}
\includegraphics[height=1  in]{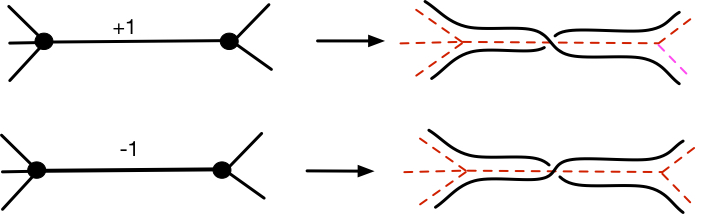}
\caption{Crossing convention}
\label{crossing}
\end{center}
\end{figure}

\begin{figure}
\begin{center}
\includegraphics[height=1.6 in]{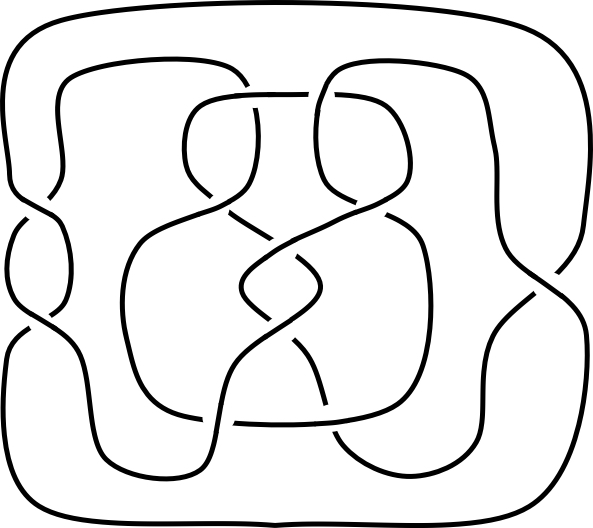}\quad \includegraphics[height=1.6  in]{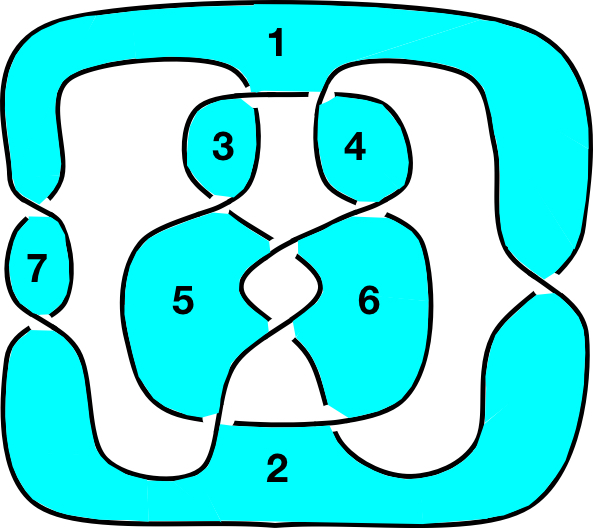}\quad \includegraphics[height=1.6  in]{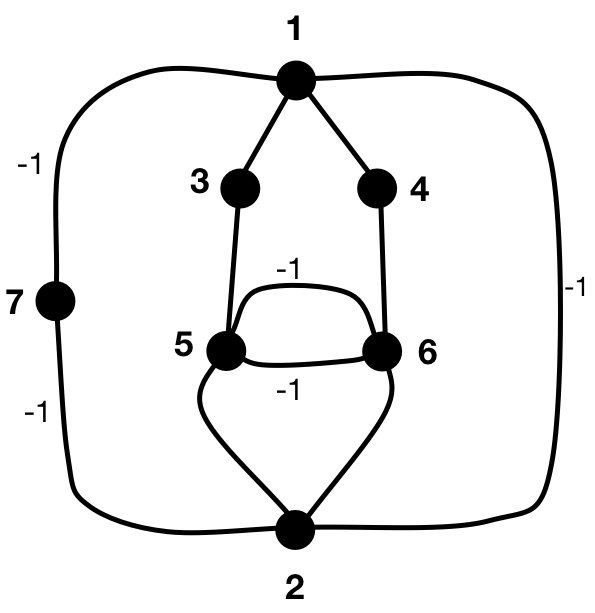}
\caption{Knot $k$; checker-board shaded diagram;  Tait graph}
\label{knot}
\end{center}
\end{figure}

Graph theorists recognize $M$ as the  \emph{Laplacian matrix} of $G$. By  Kirchhoff's celebrated Matrix Tree Theorem (see for example Theorem 2 of \cite{CK78}) any principal minor of $M$ is equal to the tree weight $\w_G$, so ${\rm Det}(\ell)=|\w_G|$. More generally, the All Minors Matrix Tree Theorem states that if $\g$ and $\g'$ are subsets of $V_G$ of equal cardinality, as in section \ref{intro}, the minor obtained by deleting the rows  and columns corresponding to $\g$ and $\g'$, respectively, is equal to $\pm \w_{G, \g,\g'}$. The sign is easily computed (see \cite{CK78}).

\begin{theorem}\label{krebes} \cite{K99} Let $\ell$ be a link. If a tangle $t$ embeds in $\ell$, then any common divisor of ${\rm Det}(n(t))$ and ${\rm Det}(d(t))$ divides ${\rm Det}(\ell)$. \end{theorem}

\begin{figure}
\begin{center}
\includegraphics[height=1.6 in]{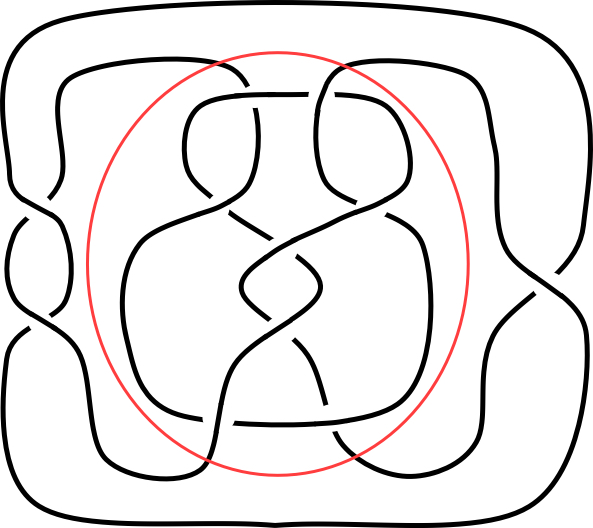}\quad \includegraphics[height=1.6  in]{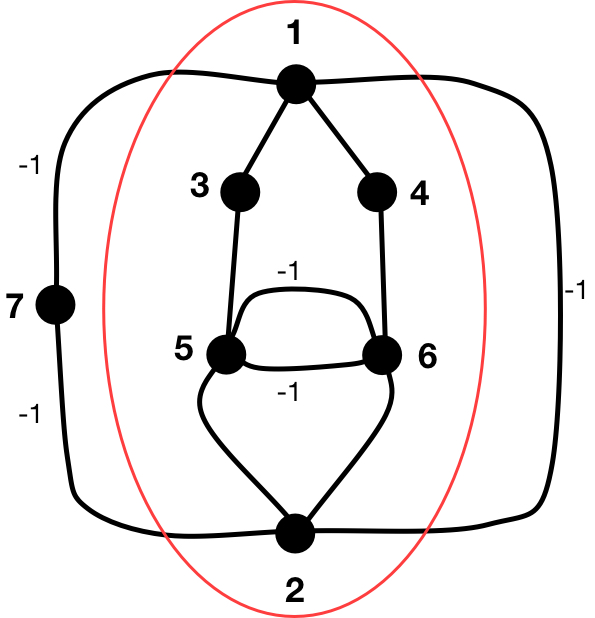}\quad  \includegraphics[height=1.6  in]{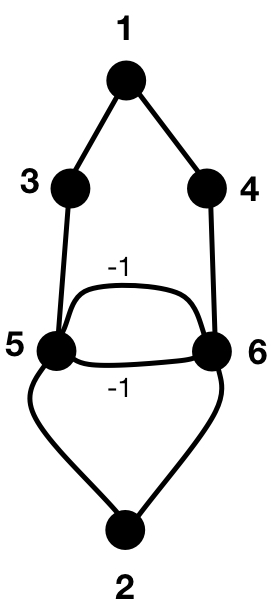}
\caption{Embedded tangle (circled, left); corresponding subgraph (center and right)}
\label{tangle}
\end{center}
\end{figure}

Krebes's proof of Theorem \ref{krebes} in \cite{K99} used the Kauffman bracket and skein theory. A shorter proof involving covering spaces and homology was given later by D. Ruberman \cite{R00}. Here we see Krebes's theorem as a direct consequence of Proposition \ref{main}. 

\begin{proof} Let $C$ be a circle that meets the diagram $D$ of the link $\ell$ transversely, enclosing a diagram of the tangle $t$.  The endpoints of $t$ divide $C$ into four segments.  We may checkerboard shade the diagram so that the faces containing the top and bottom segments of $C$ are shaded. We first suppose that these regions are distinct, and label them $1, 2$. 
Extend the numbering to the shaded faces that are enclosed by $C$, and denote by $H$ the subgraph of the Tait graph $G$ of $\ell$ containing the vertices and edges between them (as in the right-hand side of Figure \ref{tangle}). Then $H$ is a Tait graph of $n(t)$, so its Laplacian matrix is an unreduced Goeritz matrix for $n(t)$ and ${\rm Det}(n(t))=|\w_H|$.  The graph $G$ is the union of $H$ with a graph $K$ that meets $H$ in $\{1,2\}$.

Amalgamating the vertices $1, 2$ of $H$ produces the Tait graph of $d(t)$. Hence we obtain an unreduced Goeritz matrix for $d(t)$ by adding the second row and column of the Laplacian matrix of $H$ to the first row and column, respectively, and then deleting them. The absolute value of the first principal minor of the result is ${\rm Det}(d(t))$. However, deleting from $H$ the first two rows and columns and then taking the determinant produces the same end-result which, by the All Minors Matrix Tree Theorem, is $\w_{H, \g}$, where $\g =  \{1,2\}$. 

We have shown that $|\w_H|= {\rm Det}(n(t))$ and also $|\w_{H, \g }| = {\rm Det}(d(t))$, where $\g= \{1,2\}$. Since $|\w_G| = {\rm Det}(\ell)$, the desired conclusion now immediately follows from Proposition \ref{main} (i). 

We must still deal with the case where instead of distinct regions 1, 2 as above we have a single region 1.  In this case $G$ is the join of $H$ and $K$ at 1, and $H$ is the Tait graph of $d(t)$. Since $\w_G = \w_H \cdot \w_K$ we conclude that ${\rm Det}(d(t))$ divides ${\rm Det}(\ell)$.
\end{proof}

\noindent {\bf Example \ref{ex}} (continued). The left-hand side of Figure \ref{tangle} shows a tangle $t$ embedded in the knot $k$. In the center is the Tait graph $G$ that appeared previously, with a subgraph $H$ corresponding to $t$ circled in red. The subgraph is shown in the right-hand side of the diagram. 
Both ${\rm Det}(n(t))=25$ and ${\rm Det}(d(t))=30$ are easily computed from the matrix $M$. By Theorem \ref{krebes} 
any knot or link in which $T$ embeds has determinant divisible by $5$. One such knot is $k$ (above), which 
has determinant 25.

\bigskip

Corollary 5.2 (ii) of \cite{PSW05} uses algebraic topology to generalize Krebes's theorem for any $2n$-tangle. There the any common divisor of the determinants of all of its closures divides the determinant of any link in which the $2n$-tangle embeds. Part (ii) of Proposition \ref{main} gives an elementary proof of the result for $n=3$.  For this one considers the five 
tangle closures of the $6$-tangle. The proof is similar to that of Theorem \ref{krebes} and we leave it to the reader. 

\begin{theorem} \cite{PSW05} Let $\ell$ be a link. If a $6$-tangle $t$ embeds in $\ell$, then any common divisor of the determinants of the five closures of $t$ divides ${\rm Det}(\ell)$. \end{theorem}

\section*{Acknowledgement}  The authors are grateful to David Krebes for insightful suggestions concerning an earlier version of this article.

\bigskip

\ni Department of Mathematics and Statistics,\\
\ni University of South Alabama\\ Mobile, AL 36688 USA\\
\ni Email: 
\ni  silver@southalabama.edu\\
\ni swilliam@southalabama.edu
\end{document}